\documentclass[11pt]{article}

  \usepackage{natbib,ae,amsmath,amssymb,graphicx,color}
\usepackage{enumerate}
  
\usepackage{amsmath}
\usepackage{amsfonts}
\usepackage{amsthm}
\usepackage{graphicx}
\usepackage{geometry}                
\usepackage{amssymb}
\usepackage{epstopdf}
\newcounter{figian}
\newtheorem{lemma}{Lemma}[section] 
\newtheorem{propos}[lemma]{Proposition}

\newtheorem{theorem}[lemma]{Theorem}
\newtheorem{cor}[lemma]{Corollary}

\newtheorem{defin}[lemma]{Definition}

\newtheorem{conjec}[lemma]{Conjecture}

\newcommand{\nquad}{\kern-60pt}

\title{\textbf{Limits to measurement in experiments governed by algorithms}}
\author{\textbf{Edwin J.
  Beggs\dag, Jos\'e F\'elix Costa\ddag\ and J. V.\ Tucker\dag}}

\begin{document}

% \label{firstpage}
   
\maketitle
   
   \dag School of Physical Sciences, Swansea University,  Singleton Park, Swansea, SA2 8PP, Wales, United Kingdom
   
   \ddag Instituto Superior T\'{e}cnico, Technical University of Lisbon, Lisbon, Portugal \\ and Centro de Matem\'atica e Aplica\c c\~ oes Fundamentais, University of
  Lisbon, Lisbon, Portugal

\label{firstpage}
\maketitle
\begin{abstract} 
\noindent We pose the following question: \textit{If a physical experiment were to be completely controlled by an algorithm, what effect would the algorithm have on the physical measurements made possible by the experiment?}

In a programme to study the nature of computation possible by physical systems, and by algorithms coupled with physical systems, we have begun to analyse (i) the algorithmic nature of experimental procedures, and (ii) the idea of using a physical experiment as an oracle to Turing Machines. To answer the question, we will extend our theory of experimental oracles in order to use Turing machines to model the experimental procedures that govern the conduct of physical experiments. First, we specify an experiment that measures mass via collisions in Newtonian Dynamics; we examine its properties in preparation for its use as an oracle. We start to classify the computational power of polynomial time Turing machines with this experimental oracle using non-uniform complexity classes. Second, we show that modelling an experimenter and experimental procedure algorithmically imposes a limit on what can be measured with equipment. Indeed, the theorems suggest a new form of uncertainty principle for our knowledge of physical quantities measured in simple physical experiments. We argue that the results established here are representative of a huge class of experiments. 
\end{abstract}

\section{Introduction}\label{sec:intro}

Consider performing a simple experiment to measure a physical quantity. The quantity, the method, the equipment and the exact experimental procedure that together constitute the experiment are based upon, or belong to, a physical theory --- supposing the experiment to be sufficiently simple to involve only one theory. The experiment is carried out by a technician following the experimental procedure. The experimental procedure is a sequence of instructions that initialise and set parameters, control and observe the equipment, record, calculate and display results. A more modern image would introduce a computer to help the technician manage the experiment. Indeed, it seems a small step to imagine the experiment performed \textit{completely} by a computer.  We formulate the following question:
\begin{center}
\textit{If a physical experiment were to be completely controlled by an algorithm, what effect would the algorithm have on the physical measurements made possible by the experiment?}
\end{center}
This is not a small step for theorists. Notice that the technician does not belong to the physical theory: the concept plays only an informal and unanalysed role in our understanding of measurement. The role of the technician is to implement a procedure by following a sequence of instructions that are precisely formulated in terms of the physical theory. Now, the sequence of instructions resembles an algorithmic procedure. Therefore, the idea of replacing the informal concept of a technician by the formal concept of an algorithm enables us to \textit{contemplate} a new theoretical analysis of experimentation and measurement, one in which new properties, such as the process of performing an experiment, can be analysed using the methods of computability theory. In this paper we will show how the theory of algorithms can be used to make a deeper analysis of experiments and measurements.

For many years we have been thinking about what can be computed by experimenting with physical systems. In a series of papers \cite{beggstucker:06,beggstucker:07a,beggstucker:07b,beggsetal:08b,beggsetal:08c,beggsetal:08f}, we have begun to develop a methodology and mathematical theory to study the nature and limits of computations made possible by (i) physical systems in isolation and (ii) physical systems combined with algorithms. With the aims of 

(i) understanding the roles of data representations, procedures and equipment, and 

(ii) determining the computational powers of physical systems, 
\\
we introduced four principles in \cite{beggstucker:06,beggstucker:07a,beggstucker:07b} and, later, a fifth involving the idea of using a physical experiment as an oracle to a Turing machine in \cite{beggsetal:08c,beggsetal:08f}.  We began the development of a theory of physical oracles by analysing a specific experiment and its use as an oracle to a Turing machine; the experiment measured position or distance using Newtonian Dynamics and the oracle was termed the \textit{scatter machine experiment (SME)}. In the theory of physical oracles, the Turing Machine interacts with an experiment via queries that become quite complicated and need to be managed by a \textit{protocol}. We classified computational power using non-uniform complexity classes: in polynomial time the physical oracle SME allows Turing machines to compute non-uniform complexity classes such as $P/poly$ and $BPP//log*$, depending upon assumptions about precision of data. Our early experimental case studies revealed a great deal theoretically.

To attempt to answer the question here we extend our methodology by adding a new principle (Principle 6), which defines our approach. Instead of focussing on the power of the oracle to \textit{boost} the computational power of the algorithm, we focus on the aim of 

(iii) determining the power of the algorithm to \textit{govern and control} the experiment. 
\\
We show that the extended mathematical theory of physical oracles accommodates both aims and can be used for two purposes:
\begin{center}
Boosting computation $\longleftarrow$ Turing machine + protocol + experiment $\longrightarrow$ Controlling measurement.
\end{center}
In this paper, we model an experimenter or technician following an experimental procedure by means of a Turing machine. There is a similarity between an experimental procedure and that of an algorithmic procedure for they are both sequences of instructions, though the instructions are defined by radically different theories. Thus, the model of the technician following instructions is reminiscent of Turing's model of a human calculator performing a calculation. The extension to the theory of physical oracles leads to a computational model of measurement with significant properties, including: 

(a) the time taken to perform experimental measurements is a function of accuracy; 

(b) the existence of quantities that cannot be measured by any experimental procedure applied to equipment. 
\\
The model uncovers an new form of indeterminacy and suggests an uncertainty principle that imposes a limit to our knowledge of the values of the physical quantities involved in physical experiments in Newtonian mechanics. Our methods suggest that indeterminacy is a general property of experimental measurements.

First, we specify a new experiment that measures an unknown mass in Newtonian Dynamics. The experiment uses a procedure involving controlled collisions between test particles and the unknown mass. We call the experiment the \textit{collider machine experiment (CME)}. The time taken to perform the experiment is \textit{exponential} in the desired accuracy of the measurement. In this property and its consequences, the new experiment is more representative of physical measurement generally. Timing and scheduling plays a prominent role in the extended theory.

Following the methods of \cite{beggsetal:08c,beggsetal:08f}, we start to classify the computational power of this experimental oracle using non-uniform complexity classes. However, the lower bound results here differ from those of the experiment in \cite{beggsetal:08c,beggsetal:08f} for reasons that are significant (see theorems in Section \ref{sec:collidercomplexity} and the discussion in Section \ref{cme&sme}):

\begin{theorem}
Turing machines equipped with the collider machine experiment as an oracle can compute the following complexity classes in polynomial time: assuming the mass of the test particle can be set with

Infinite precision: $P/log*$.

Finite unbounded precision: $P/log*$.

Finite fixed precision: $BPP//log*$,
\\
assuming, in the last case, a uniform distribution for the mass of the test particle.
\end{theorem} 

Secondly, we reverse the direction of the application of the theory to address the main question. Modelling experimental procedures algorithmically lead to results that challenge a fundamental assumption of the classical physical world, namely, that values of a physical quantity $\mu$ belong in experimental procedures as patterns of the form $\mu \pm \Delta \mu$, where theoretically the value $\Delta \mu$ can be made as small as we wish. For a century, this assumption confronted the quantum world wherein measurements cannot be made with errors as small as wish.  We show that the mathematics of computation imposes limits on what can be measured (Theorem \ref{nonmeasurable}):

\begin{theorem}
There are uncountably many masses $\mu$ such that for every experimental procedure governing the CME it is only possible to determine finitely many digits of $\mu$, even allowing arbitrary long run times for the procedure.
\end{theorem}

Although not every mass can be measured --- however one varies the procedure --- we can show the following:

\begin{theorem}
There is a universal experimental procedure that measures every mass $\mu$ that can be measured by some procedure using the CME. 
\end{theorem}

This is the third in our series of basic papers on physical oracles: the reader is recommended the first \cite{beggsetal:08c} for background and essential technical material (e.g., non-uniform complexity; time constructible functions etc.). In \cite{beggsetal:08d,beggsetal:09a} we argue that algorithms with physical oracles occur naturally in Physics. In \cite{beggsetal:09c} we have begun to study relativisations of the $P=NP$ problem to physical oracles.

\section{Computation and physical systems}\label{Computation_and_physical_systems}
We summarise five methodological principles for determining the computational power of physical systems and introduce the new principle determining the power of experimental procedures.  

\subsection{Methodological principles for experimental computation} \label{Methodological _principles_experimental}

To explore what may be computed by experimenting with physical systems we must theorise from first principles that are independent of theory of algorithms. For conceptual clarity, and mathematical precision and detail, we proposed, in \cite{beggstucker:06, beggstucker:07a}, the following four principles and stages for an investigation of any class of experimental computations:

\medskip\noindent
{\bf Principle 1. Defining a physical subtheory:} \textit{Define precisely a subtheory $T$ of a physical theory and examine experimental computation by the
  systems that are valid models of the subtheory $T$.}

\medskip\noindent
{\bf Principle 2. Classifying computers in a physical theory:} \textit{Find systems that are models of $T$ that can through experimental computation implement
  specific algorithms, calculators, computers, universal computers and hyper-computers.}

\medskip\noindent
{\bf Principle 3. Mapping the border between computer and hyper-computer in physical theory:} \textit{Analyse what properties of the subtheory $T$ are the source of computable and non-computable behaviour and seek necessary and sufficient conditions for the systems to implement precisely the algorithmically computable functions.}

\medskip\noindent
{\bf Principle 4. Reviewing and refining the physical theory:} \textit{Determine the physical relevance of the systems of interest by reviewing the truth or valid scope of the subtheory.  Criticism of the system might require strengthening the subtheory $T$ in different ways, leading to a portfolio of theories and examples.}

\medskip\noindent
Our methodology requires a careful formulation of a physical theory $T$, which we have discussed at length elsewhere \cite{beggstucker:06,beggstucker:07a}. From the theory $T$ are derived the central concepts of \textit{experimental procedures} and \textit{equipment}. We need to control $T$ to lay bare all the concepts and technicalities to be found in examples and to classify their computational behaviour. The theory $T$ is everywhere and so we are actually studying $T$-{\it computability} and $T$-{\it computational complexity}. Our approach has been applied in a new discussion of the physical basis of the Church-Turing Thesis in Ziegler \cite{ziegler:08}. 

\subsection{Methodological principles for combining experiments and algorithms} \label{Combining}

Next, we extend our methodology to consider the {\it interaction} between experiments and algorithms. 

First, consider using an experiment as a component to boost the performance of an algorithm or class of algorithms. In this case, computations involve
some form of \textit{protocol} for exchanging data between physical system and algorithm. A simple general way to do this is to choose an algorithmic model and incorporate the experiment as an oracle. There are many algorithmic models but the advantage of choosing Turing machines is their rich theory of computational complexity.

Suppose we wish to study the complexity of computations by \textit{Turing machines with experimental oracles}.  Given an input string $w$ over the alphabet of the Turing machine, in the course of a finite computation, the machine will generate and receive a finite sequence of oracle queries and answers. Specifically, as the $i$-th question to the oracle, the machine generates a string that is converted into a rational number $z_{i}$ and used to set an input parameter, a physical quantity, $p_{i}$ to the equipment.  The experiment is performed and, after some delay, returns as output a rational number measurement, or qualitative observation, which is converted into a string or state for the Turing machine to process. The Turing machine may pause while waiting for the oracle, which may take quite some time. This time must be measured by a clocks belonging to the Turing machine and the protocol. In \cite{beggsetal:08c}, we introduced the fifth principle:
\\
\\
{\bf Principle 5. Combining experiments and algorithms:} \textit{ Use a physical system as an oracle in a model of algorithmic computation, such as Turing machines.  Determine whether the subtheory $T$, the experimental computation, and the protocol extends the power and efficiency of the algorithmic model.}
\\
\\
Now consider experiments and the way they measure physical quantities.  At the heart of our conception of an experiment is an \textit{experimental procedure} that is applied to \textit{equipment}. As we proposed in the Introduction, an experimental procedure is sequence of instructions that act on and observe a physical system, and are to be performed step by step. More precisely an experimental procedure that is a text consisting of instructions, commands and rules derived from the physical theory $T$. The idea that an experiment is reproducible and subject to debate requires precise specifications and rigorous reasoning, all of which depends upon $T$. The analogy between experimental procedures and algorithmic procedures and programs was noticed in our earliest investigations --- a language for Newtonian experimental procedures was described in \cite{beggstucker:08a}.  

Next, we imagine a human experimenter or technician performing an experimental procedure, reading the text, following instructions, commands and rules defined by a physical theory $T$. The image can be compared with Turing's analysis of a human computing with pencil and paper, following instructions, commands and rules, which is the fundamental intuition behind the Turing machine. If the $T$-instructions are coded as strings then the Turing machine can serve an abstract model of the technician following the procedure. Of course, today, many experiments are fully computer controlled. Here we introduce:
\\
\\
{\bf Principle 6. Algorithms controlling experiments:} \textit{ Use a model of algorithmic computation, such as Turing machines, to control a physical system.  Determine whether the subtheory $T$, the experimental procedure and equipment, and the protocol extends or limits the accuracy and efficiency of the physical experiment to make measurements.}
\\
\\ 
All four components have effects. There are many variations when modelling even simple equipment, possibly requiring a choice of assumptions in the theory $T$. Recently, we remodelled the equipment of the SME in \cite{beggsetal:08c} and obtained dramatically different results on rounding a pointed wedge. In addition, the models of measurement introduce time in a natural quantitative way. The role of time is somewhat neglected in formal theories of measurement after Hempel \cite{hempel:52}; see \cite{suppesetal:09}.

Our experiments are idealised. They are designed to uncover fundamental ideas and properties of less perfect systems. Often, our results are intended to be a theoretical best case, where introducing more realism would only serve to make the results worse. Like the Turing machine, the experiments capture the essential features of examples, and map the limit of physical reality. Indeed, the ontology of our {\em gedankenexperimente} means that there is no fundamental difference between the ideal experiments defined by physical theory and the Turing machine.

\section{The collider machine experiment}\label{sec:collider}

In this and the next sections, we describe an experiment that uses elastic collisions to measure the unknown (inertial) mass of a particle to arbitrary accuracy. 

\subsection{Theory}

The elastic collision between two particles on a line is dictated by two basic laws of Physics: the conservation of {\em linear momentum} and the conservation of {\em kinetic energy}, These laws are underpinned by the homogeneity of Newtonian space and time, respectively, and both can be derived from the three laws of Newtonian dynamics, which we make take to be the underlying physical theory $T$ (cf. \cite{hensuptar:59}).

\subsection{Experimental equipment and its behaviour}
In the one dimensional collision, the center of mass of the two particles are in the same line of motion. Let $m$ and $\mu$ be the masses of the two particles. We will assume that the particle of ``unknown'' mass $\mu$ is always at rest before the collision, and that the test particle of mass $m$ is projected along the line towards the particle of unknown mass $\mu$ with speed $u = 1.0 \ (\pm \ \varepsilon) \ ms^{-1}$, e.g. with $0 \leq \varepsilon \leq 0.1$\footnote{This error margin in the initial speed of the test particle of mass $m$ means that precision in speed does not matter for this experiment.}. After the collision the particle of mass $m$ acquires the speed $v_m$ and the particle of mass $\mu$ is projected forward with speed $v_{\mu}$.

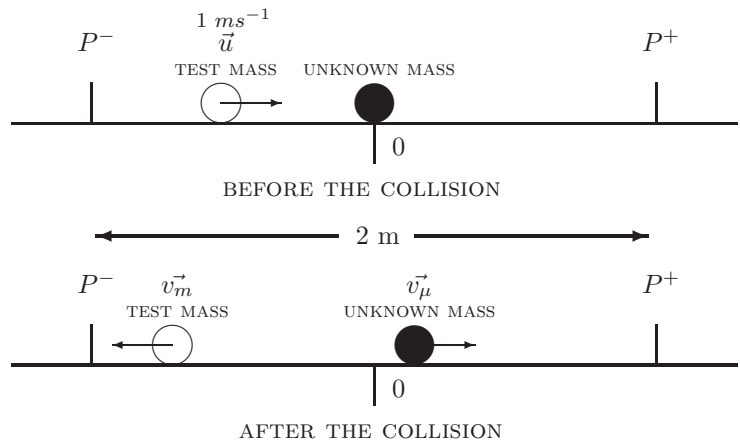
\begin{figure}[ht]

\setlength{\unitlength}{2000sp}
\begingroup\makeatletter\ifx\SetFigFont\undefined
\gdef\SetFigFont#1#2#3#4#5{\reset@font\fontsize{#1}{#2pt}
\fontfamily{#3}\fontseries{#4}\fontshape{#5}\selectfont}
\fi\endgroup

\begin{picture}(6000,6000)(0,0)

\thicklines

\put(2000,4000){\line(0,1){500}}
\put(9000,4000){\line(0,1){500}}

\put(1000,4000){\line(1,0){9000}}

\put(5050,2600){\vector(-4,0){3000}}
\put(6000,2600){\vector(4,0){2900}}

\put(2000,1000){\line(0,1){500}}
\put(9000,1000){\line(0,1){500}}

\put(1000,1000){\line(1,0){9000}}

\put(5500,3500){\line(0,1){500}}
\put(5500,500){\line(0,1){500}}

\put(1000,1000){\line(1,0){9000}}

\thinlines

\put(5500,4250){\circle*{500}}
\put(3600,4250){\circle{500}}
\put(3600,4250){\vector(1,0){750}}

\put(6000,1250){\circle*{500}}
\put(3000,1250){\circle{500}}
\put(3000,1250){\vector(-1,0){750}}
\put(6000,1250){\vector(1,0){750}}

\put(5150,2500){\makebox(0,0)[l]{\smash{{\SetFigFont{10}{14.4}{\rmdefault}{\mddefault}{\updefault}{2 m}}}}}

\put(3500,3100){\makebox(0,0)[l]{\smash{{\SetFigFont{10}{14.4}{\rmdefault}{\mddefault}{\updefault}{{\sc before the collision}}}}}}
\put(3700,100){\makebox(0,0)[l]{\smash{{\SetFigFont{10}{14.4}{\rmdefault}{\mddefault}{\updefault}{{\sc after the collision}}}}}}

\put(5500,4600){\makebox(0,0)[c]{\smash{{\SetFigFont{10}{14.4}{\rmdefault}{\mddefault}{\updefault}{{\sc \scriptsize unknown mass}}}}}}
\put(3600,4600){\makebox(0,0)[c]{\smash{{\SetFigFont{10}{14.4}{\rmdefault}{\mddefault}{\updefault}{{\sc \scriptsize test mass}}}}}}
\put(3600,4900){\makebox(0,0)[c]{\smash{{\SetFigFont{10}{14.4}{\rmdefault}{\mddefault}{\updefault}{$\vec{u}$}}}}}
\put(3700,5200){\makebox(0,0)[c]{\smash{{\SetFigFont{10}{14.4}{\rmdefault}{\mddefault}{\updefault}{{\scriptsize $1 \ m s^{-1}$}}}}}}

\put(6000,1600){\makebox(0,0)[c]{\smash{{\SetFigFont{10}{14.4}{\rmdefault}{\mddefault}{\updefault}{{\sc \scriptsize unknown mass}}}}}}
\put(3000,1600){\makebox(0,0)[c]{\smash{{\SetFigFont{10}{14.4}{\rmdefault}{\mddefault}{\updefault}{{\sc \scriptsize test mass}}}}}}
\put(3000,1900){\makebox(0,0)[c]{\smash{{\SetFigFont{10}{14.4}{\rmdefault}{\mddefault}{\updefault}{$\vec{v_m}$}}}}}
\put(6000,1900){\makebox(0,0)[c]{\smash{{\SetFigFont{10}{14.4}{\rmdefault}{\mddefault}{\updefault}{$\vec{v_{\mu}}$}}}}}

\put(8700,4900){\makebox(0,0)[l]{\smash{{\SetFigFont{10}{14.4}{\rmdefault}{\mddefault}{\updefault}{$P^+$}}}}}
\put(1700,4900){\makebox(0,0)[l]{\smash{{\SetFigFont{10}{14.4}{\rmdefault}{\mddefault}{\updefault}{$P^-$}}}}}

\put(8700,1900){\makebox(0,0)[l]{\smash{{\SetFigFont{10}{14.4}{\rmdefault}{\mddefault}{\updefault}{$P^+$}}}}}
\put(1700,1900){\makebox(0,0)[l]{\smash{{\SetFigFont{10}{14.4}{\rmdefault}{\mddefault}{\updefault}{$P^-$}}}}}

\put(5600,3600){\makebox(0,0)[l]{\smash{{\SetFigFont{10}{14.4}{\rmdefault}{\mddefault}{\updefault}{$0$}}}}}
\put(5600,600){\makebox(0,0)[l]{\smash{{\SetFigFont{10}{14.4}{\rmdefault}{\mddefault}{\updefault}{$0$}}}}}

\end{picture}

\caption{{\sc collider machine} experiment.}\label{variation0}

\end{figure}

Since the conservation of energy reduces to the conservation of kinetic energy in the case of elastic collisions in the horizontal plane, the collision is described by the equations: 
\begin{equation}\label{momentum}
m u = m v_m + \mu v_{\mu},
\end{equation}
\begin{equation}\label{kinetic}
\frac{1}{2} m u^2 = \frac{1}{2} m v_m^2 + \frac{1}{2} \mu v_{\mu}^2,
\end{equation}
\noindent that can be solved for $v_m$ and $v_{\mu}$: 
\begin{equation}\label{vm}
v_m = \frac{m - \mu}{m + \mu} \ u,
\end{equation}
\begin{equation}\label{vmu}
v_{\mu} = \frac{2m}{m + \mu} \ u.
\end{equation}

From these formulae we see that after a collision: 

\noindent (a) if $m < \mu$, then the test particle move backwards; 

\noindent (b) if $m > \mu$, then the test particle will move forward; and 

\noindent (c) if $m = \mu$, then the test particle of mass $m$ comes to rest and the particle of unknown mass $\mu$ is projected forward with the previous value of the speed of the test particle.

An experiment can be designed to measure the unknown mass $\mu$, using test particles of known mass $m$ projected at approximately the same speed $u$; the procedure is given in the next section.

We establish the convention that the particle of unknown mass is placed at the origin of coordinates and points $P^- \equiv -1 \ m$ and $P^+ \equiv +1 \ m$ are the {\em flags} of the experimenter's observations: when the test particle is seen crossing the points $P^-$ or $P^+$ the experiment terminates. If the test mass crosses the flag $P^-$ then we have $m<\mu$ (as depicted in Figure Fig.\ 2), and if it crosses the flag $P^+$, we have $m>\mu$.

Consider some tolerances in the design of this equipment. Among properties that are largely irrelevant, or where errors can be tolerated, are: the (finite) distance between the two flags; the precision of the placement of the flags; the error in placing the particle of the unknown mass at the origin; the initial speed of the test particle (let us say approximately $1 \ m s^{-1}$); and the error in the speed of the unknown mass. Note that the observed velocities of the particles after the collision, after crossing one or both the flags, are irrelevant.

However, properties and quantities that are relevant include: the one dimensional character; the masses of the unknown particles are in the range $(0,1)$; the particle of unknown mass $\mu$ is at rest; and that the collisions are elastic.

\subsection{Timing the experiment}

Looking closely at the experiment, we find an experimental barrier: the time for the test particle crossing the distance of $1 \ m$ after the collision is given by 
\begin{equation}\label{time}
t_m \,=\, \frac{1}{u} \left| \frac{m + \mu}{m - \mu} \right|\,=\,
 \frac{1}{u} \left|1+ \frac{2\, \mu}{m - \mu} \right|
\ ,
\end{equation}
so the time taken to complete the experiment is of the order
\begin{equation}\label{timesim}
t_m\, \propto\, \left| \frac{1}{m - \mu} \right|\ .
\end{equation}
\noindent For the values we will take for the masses and initial speed, the time is of the order of 
\begin{equation}\label{ctimesim}
\frac{A}{|m - \mu|} \ \leq \ t_{exp} \ \leq \ \frac{B}{|m - \mu|} \ ,
\end{equation}
for some constants $A$ and $B$.

Suppose we wanted to measure the unknown mass with accuracy $|m - \mu| \le 2^{-n}$. Then, substituting into Inequality \ref{ctimesim}, we need exponential time:
\begin{equation}\label{cdddvsducv}
A.2^{n} \ \leq \ t_{exp}
\end{equation}
Thus, the time taken goes to infinity as the test particle's mass $m$ approaches the unknown mass $\mu$. If we have $m=\mu$ then we will wait forever for the result. 

\section{Experimental procedures and protocols}\label{protocoldiscuss}

To find the unknown mass, which we assume is known to lie in the interval $(0,1)$, we need an experimental procedure. The experimental procedure we use is based on a bisection method and, ultimately, must be represented by a Turing machine.  For the experimental procedure to operate the equipment via oracle queries, we must specify new mechanisms and a protocol. 

\subsection{Oracle queries and setting initial conditions}

The Turing machine is connected to the collider experiment CME in the same way as it would be connected to an oracle: we replace the query state with a {\em shooting state} ($q_s$), the ``yes'' state with a {\em lesser state} ($q_l$), and the ``no'' state with a {\em greater state} ($q_g$). The resulting computational device is called the {\em (analogue-digital) collider machine experiment}.

In order to carry out an experiment with the equipment, the machine will write a word $z$ in the query tape and enter the shooting state. The word $z$ codes for a dyadic rational mass $m$ of the test particle. In the \textit{shooting state} the machine prepares and fires a test particle of mass $m$
as detailed above. The experiment continues until the test particle crosses one of the flags $P^\pm$, and then returns to a state $q_l$, for $m<\mu$, or a state $q_g$, for $m>\mu$, of the Turing machine.

Technically, this word $z$ will either be ``$1$'', or a binary word beginning with $0$. We will use $z$ ambiguously to denote both a word $z_1 \ldots z_n \in \{1\} \cup \{0s: s \in \{0,1\}^*\}$ and the corresponding dyadic rational $\sum_{i=1}^n 2^{-i+1} z_i \in [0,1]$. In this case, we write $|z|$ to denote $n$, i.e., the size of $z_1 \ldots z_n$. 
The resulting computational device is called the {\em analogue-digital collider machine}.

Consider the precision of the experiment. When measuring the output state the situation is simple: either the test particle of mass $m$ crosses $P^-$ or it crosses $P^+$ (or,  after time some timeout, no test particle is detected). Errors in observation do not arise. However, there are different postulates for the precision of the experiment, in order of decreasing strength:

\begin{defin}\label{CME}
The CME is error-free if the mass of test particle can be set exactly to any given dyadic rational number. The CME is error-prone with arbitrary precision if the mass of test particle can be set to within a non-zero, but arbitrarily small, dyadic precision. The CME is error-prone with fixed precision if there is a value $\varepsilon > 0$ such that the mass of test particle can be set only to within a given precision $\varepsilon$.
\end{defin}

\subsection{The bisection method: basic ideas and constraints}

The Turing machine can choose and fire the test mass for the CME. For simplicity, we shall assume that the masses are set exactly, without error, according to the dyadic rational coded on the query tape; our account applies to the error-prone cases mutato nomine. 

The bisection method  for the CME is roughly this.  We start by two dyadic rational test masses $m_1 \ (= \ 0 \ Kg)$\footnote{Although this has no physical meaning in the setting of Classical Mechanics, this fact is irrelevant as will soon become apparent.} and $m_2 \ (= \ 1 \ Kg)$ such that that, after the collision, the test particle is projected backward and forward, respectively. We know that the unknown mass $\mu$ is in the interval $(m_1,m_2)$. Then we try with the mass $m = \frac{m_1 + m_2}{2}$. If $m$ is reflected, then we know that $m < \mu < m_2$ and we make the assignment $m_1 := m$, else if $m$ is projected forward, then we know that $m_1 < \mu < m$ and we make the assignment $m_2 := m$. In principle, we will get {\em a binary digit of the unknown mass per collision}. 

However there are complications. There are potential problems with in the case where the unknown mass is dyadic and there is the matter of timing. As the method proceeds, we will have dyadic fractions which are described by longer and longer words. It is reasonable that setting up the apparatus for longer words would require more time (e.g., the increased time just to read the word). There are physical reasons inherent in the CME for just why some experiments take longer times than others. 

%The outcome of the experiment is determined by which flag the test mass crosses. However the speed of the test mass is not a constant, as illustrated in Fig.\ 2.
%\\
%\begin{figure}[ht]\label{figurespeedcme}

%\setlength{\unitlength}{2000sp}
%\begingroup\makeatletter\ifx\SetFigFont\undefined
%\gdef\SetFigFont#1#2#3#4#5{\reset@font\fontsize{#1}{#2pt}
%\fontfamily{#3}\fontseries{#4}\fontshape{#5}\selectfont}
%\fi\endgroup

%\begin{picture}(6000,5000)(0,0)

%\thicklines

%\put(3000,1025){\line(0,1){4000}}

%\put(2125,3000){\line(1,0){8000}}

%\thinlines

%\put(3000,1125){\line(1,0){7000}}

%\put(3000,5000){\line(1,0){7000}}

%\qbezier(3000,5000)(6500,5000)(7000,3000)
%\qbezier(7000,3000)(7500,1000)(10000,1150)

%\put(3950,5050){\makebox(0,0)[rt]{\smash{{\SetFigFont{10}{14.4}{\rmdefault}{\mddefault}{\updefault}{velocity}}}}}
%\put(9500,2600){\makebox(0,0)[l]{\smash{{\SetFigFont{10}{14.4}{\rmdefault}{\mddefault}{\updefault}{$m$}}}}}
%\put(6500,3050){\makebox(0,0)[l]{\smash{{\SetFigFont{10}{14.4}{\rmdefault}{\mddefault}{\updefault}{$\mu$}}}}}
%\put(2600,4900){\makebox(0,0)[l]{\smash{{\SetFigFont{10}{14.4}{\rmdefault}{\mddefault}{\updefault}{$u$}}}}}
%\put(2350,900){\makebox(0,0)[l]{\smash{{\SetFigFont{10}{14.4}{\rmdefault}{\mddefault}{\updefault}{$-u$}}}}}

%\end{picture}

%\vskip-0.2in

%\caption{Velocity {\em versus} mass $m$ for the test particle in the CMEt.}\label{variation2}

%\end{figure}

As the Turing machine does not know the value of $\mu$, it has no way of determining how long an experiment will take. So the questions arise: How does the programmer of the Turing machine deal with this problem? Shall the Turing machine wait indefinitely, or shall the programmer employ a timer or schedule and abort the experiment after a certain time? If we have a ``time out"  in an experiment, does the Turing machine program halt, or does it try to ``deal'' with the time out? 

\subsection{Time schedules and the collider machine protocol}\label{cmprotocol}

For the CME there are definite physical restrictions on choosing its protocol: There are constants $K,N>0$ so that 
for a  query $z$ of length $|z|$, giving a dyadic rational $z$ that determines the test mass $m$, the time taken to return a result is
\begin{center}
$K/|\mu-z|\pm N$.
\end{center} 
The machine returns one result if $\mu>z$, and another if $\mu<z$. If $\mu=z$ we get no result, i.e.\ we wait an infinite amount of time, unless the experiment is timed out by a process in the Turing machine. Such a time limit would mean that
no result would be given for $z$ sufficiently close to $\mu$, not just the case where $z=\mu$. We separate two cases: Suppose

\quad{(1)}\quad  \textit{There is no time limit on the oracle consultation}. 
After making a query $z$, the Turing machine will wait until it receives an interrupt from the collider machine with the result, which is either $\mu>z$ or $\mu<z$. 
At this time calculation will resume. It is important to note that, in this case, the Turing machine has no idea of the elapsed time and that the Turing machine may wait indefinitely. 

\quad{(2)}\quad \textit{There is a time limit on the oracle consultation}.  We assume the time limit has the form a function $T(|z|)$, where $z$ is the query. 
We use this as a timer in the following way: When the query $z$ is made, the Turing machine starts a timer which counts time
$T(|z|)$, and then the Turing machine looks for an answer from the collider machine, which may or may not have reached a decision by that time. The possible results are $\mu>z$, $\mu<z$ or \textit{out of time}. Note that the Turing machine
must wait for the whole time $T(|z|)$. Also note that if $z=\frac12$ for example, the Turing machine can increase the waiting time just by adding trailing $0$s to $z$, a padding that does not affect its numerical value, but affects its word length $|z|$\footnote{Padding buys time, but a polynomial protocol could not gain sufficient time for the CME by padding: to perform the experiment it would need exponential size padding, which would require exponential time.}. 

%...........................
%Consider a polynomial protocol $n^k$. For a query $z$ of size $n$, let
%$f(n)$ be the size of the padded query $z0^s$, where $f$ has to be a
%computable function. The resource of time is, in this case, $f(n)^k$ and it
%will be enough for the collider machine experiment, in the case of masses
%with regular distribution of the bit $1$, if the following relation holds:
%\\
%\begin{eqnarray*}
%A \ 2^{k' n} \ = \ f(n)^k \ ,
%\end{eqnarray*}
%\\
%from where we deduce
%\\
% \begin{eqnarray*}
%\log_2 \ A \ + \ k' n \ = \ k \ \log_2 \ f(n) \ ,
%\end{eqnarray*}
%\\
%that is
% \\
% \begin{eqnarray*}
%f(n) \ = \ B \ 2^{\frac{k'}{k} \ n} \ ,
%\end{eqnarray*}
%\\
%where $B = log_2 \ A / k$. This means that we need an exponential padding to
%work with a polynomial protocol in the collider machine experiment.
%..............................

It is reasonable to assume that the reason for increasing the word length of the query is to find the mass $\mu$ more accurately. But the time taken to find $\mu$ to more places increases at least exponentially in the number of places. If we do not allow for an increase in time which is at least exponential in the word length, we have no hope of finding $\mu$ to that length. 

It is important to note that there are two possible reactions by the Turing machine to an experiment having timed out, and both of these will be considered later in the paper: 

(1) The computation is aborted. This will be used in the definition of measurability later.

(2) The computation is continued, with a result of ``time out'' being returned by the physical oracle. This will be used in the fixed precision results later.

\subsection{The bisection method: experimental procedure}
The function $T$ specifies a timer or schedule, i.e., in each experiment, in order to read the $|m|$-bit of the mass $\mu$, $T(|m|)$ gives the number of time steps that the experimenter is prepared to wait before abandoning the experimental run.

\vspace{0.3cm}

\noindent $Bisection(T)$ ---  An experimental procedure to read the first $n$ bits of a unknown mass $\mu$ with schedule $T$\label{bisectionalg}
\begin{enumerate}
\item {\bf input} $n$ --- required precision coded by the number of places to the right of the left leading $0$;
\item $m_1 := 0$, $m_2 := 1$, $m := 0$ --- initial values with no physical significance; note $|m_1| = 0$, $|m_2| = 1$, and $|m| = 0$;
\item {\bf while} $|m| \leq n$ {\bf do}

\begin{enumerate}
\item $m := \frac{m_1 + m_2}{2}$;
\item place the particle of unknown mass $\mu \in [0, 1]$ at the origin; 
\item project test particle of mass $m$ to collide with particle of unknown mass;
\item {\bf if} test particle crosses the flag $P^-$ in time $T(|m|)$ {\bf then} $m_1 := m$; append $1$; --- it is known that $\mu \in (m,m_2)$;
\item {\bf if} test particle crosses the flag $P^+$ in time $T(|m|)$ {\bf then} $m_2 := m$; append $0$; --- it is known that $\mu \in (m_1,m)$;
\item {\bf if} no particle crosses the flags in time $T(|m|)$ {\bf then} return time out;
\end{enumerate}
\item {\bf end while};
\item  {\bf output} dyadic rational denoted by $m$.
\end{enumerate}

\vspace{0.3cm}

The bisection method is parameterised by the time schedule $T$ given to the experimenter to test for the crossing of one of the flags. It applies to each type of precision. There is the essential constraint on $T$ which has fundamental consequences.

%\begin{propos}\label{cexponentiallowerbounds}
%At each stage of the bisection algorithm, the lower bounds on the time of a single experiment with the CME are exponential in the size of the mass of the test particle.
%\end{propos}

%\noindent {\em Proof}: We know that the time taken by a single experiment is given by Inequality \ref{ctimesim} at step $n$ with $|m| = n$. Thus $\mu$ has a pattern of the form $\mu = m \pm m' \times 2^{-n'-1}$, with $m' \in [0,1]$ and $n' > n$, and $t_{exp}$ has a pattern of the form $$t_{exp} \ \sim \ \frac{K}{|m - (m \pm m' \times 2^{-n'-1})|} \ ,$$ that is,\footnote{Let $f$ and $g$ be total maps with signature $\mathbb{N} \to \mathbb{N}$. We say that $f \in \Omega(g)$ if there exists a constant $k \in \mathbb{R}$ such that, for an infinite number of values of $n \in \mathbb{N}$, $f(n) > k g(n)$.} $$t_{exp} \ \sim \ \frac{K}{|\pm m' \times 2^{-n'-1}|} \ \in \Omega(2^n) \ .$$ \hfill $\Box$

\begin{propos}\label{cexponentiallowerbounds}
At the $n$th stage of the bisection algorithm, the lower bound on the time of a single experiment with the CME is exponential in $n$.
\end{propos}

\noindent {\em Proof}: At the stage of the bisection algorithm where we try to determine the $n$th place
of $\mu$, we need to employ a test particle of mass $m$ where $|m - \mu| \le 2^{-n}$.
Then Equation \ref{cdddvsducv} gives $A.2^{n}  \leq  t_{exp} $.
 \hfill $\Box$

Thus, we have the following consequence:

\begin{propos}
The protocol that processes queries between a Turing machine and the collider takes time that is at least exponential in the size of the mass of the test particle specified by the queries.
\end{propos}

Finally, there are some observations of properties that are experimentally undecidable:

\begin{propos}\label{massundecidability}
That the test mass coincides with the given unknown mass cannot be established experimentally in finite time by the CME under any experimental procedure.
\end{propos}

\noindent {\em Proof}: According to Equation \ref{cdddvsducv}, as $m \rightarrow \mu$ through the bisection method, the time the experimenter has to wait goes to infinity, $t_{exp} \rightarrow + \infty$. If the two masses coincide, then the experimenter will never know. \hfill $\Box$

\begin{propos}\label{dyadicmass}
To know if the unknown mass is a dyadic rational cannot be established experimentally in finite time by the CME under any experimental procedure.
\end{propos}

\subsection{Computation using CME as an oracle}

Here are the definitions for deciding sets by Turing machines with the help of the CME (recall \cite{beggsetal:08c}).

\begin{defin}\label{langdetercollider}
Let $A \subseteq \Sigma^*$ be a set of words over $\Sigma$. We say that an error-free analogue-digital collider machine $M$ {\em decides} $A$ if there exists a time constructible schedule $T$ for a protocol to operate the oracle CME($\mu$) such that, for every input $w \in \Sigma^*$, $w$ is accepted if $w \in A$ and rejected when $w \notin A$. We say that $M$ {\em decides $A$ in polynomial time}, if $M$ decides $A$, and there is a polynomial $p$ such that, for every $w \in \Sigma^*$, the number of steps of the computation is bounded by $p(|w|)$.
\end{defin}

\begin{defin}\label{langprobcollider}
Let $A \subseteq \Sigma^*$ be a set of words over $\Sigma$. We say that an error-prone analogue-digital scatter machine $M$ {\em decides} $A$ if there exists a time constructible schedule $T$ for a protocol to operate the oracle CME($\mu$) and a number $\gamma < \frac{1}{2}$, such that the error probability of $M$ for any input $w$ is smaller than $\gamma$. We call {\em correct} to those computations which correctly accept or reject the input. We say that $M$ {\em decides $A$ in polynomial time}, if $M$ decides $A$, and there is a polynomial $p$ such that, for every input $w \in \Sigma^*$, the number of steps in every computation of $M$ on $w$ is bounded by $p(|w|)$.
\end{defin}

%We can define different classes of time complexity, e.g., if $\mathit{exp}$ is the class of functions of the form $\lambda n. \ 2^{kn}$, for some $k$, where $n$ is the size of the test mass expressed in binary notation (dyadic rational), then we can define $$DEXT^{\dagger} = \bigcup_{k = 1}^{+ \infty} \ \{A \subseteq \{0,1\}^*: A \ \mbox{is decided in time} \ \lambda n. \ 2^{kn}\},$$ where the dagger denotes that the time in question is physical time instead of computational time. In the same fashion the other computational classes can be defined.

\section{Preliminaries on non-uniform complexity}\label{sec:complexity}

In this paper $\Sigma$ denotes an {\em alphabet} and $\Sigma^*$ denotes the {\em set of finite words} over $\Sigma$. A {\em language} is a subset of $\Sigma^*$.  Almost always we will adopt the binary alphabet $\{0, 1\}$. 

The {\em pairing function} is the well known map $\langle -,- \rangle: \Sigma^* \times \Sigma^* \to \Sigma^*$, computable in linear time, that allows to encode two words in a single word over the same alphabet by duplicating bits and inserting a separation symbol ``$01$''. 

We recall the definition of {\em non-uniform complexity class}. By an {\em advice function} we mean any total map $f: \mathbb{N} \to \Sigma^*$. 

\begin{defin}\label{bar}
Let $\mathcal{B}$ be a class of sets and $\mathcal{F}$ be a class of advice functions. Then we define the new class $\mathcal{B}/\mathcal{F}$ as the class of sets $A$ such that there exists a set $B \in \mathcal{B}$ and an advice $f \in \mathcal{F}$ such that, for every word $x \in \Sigma^*$, $x \in A$ if, and only if, $\langle x, f(|x|) \rangle \in B$.
\end{defin}

Suppose we fix the class $\mathcal{B}$ to be the class $P$ of sets decidable by Turing machines in polynomial time. Then we still have one degree of freedom which is the class of advice functions $\mathcal{F}$ that makes $P/\mathcal{F}$. In this paper, we will work with subpolynomial advice functions, i.e., $\mathcal{F}$ is a class of functions with sizes bounded by {\em polynomials} and computable in polynomial time. Note that the advice functions are not, in general, computable but the associated bounds are computable; e.g., if the class is $poly$, then it means that any, possibly non-computable, advice function $f: \mathbb{N} \to \Sigma^*$, is bounded by a (computable) polynomial $p$ such that, for all $n \in \mathbb{N}$, $|f(n)| \leq p(n)$.

\subsection{Deterministic classes}\label{sec:deterclass}

Although the class $\mathcal{F}$ of functions is arbitrary it is useless to use functions with growth rate greater than exponential. Let $exp$ be the set of advice functions bounded in size by functions in the class $2^{O(n)}$. Then $P/exp$ contains all sets. Given this fact, we wonder if either $P/poly$ or $P/log$ (subclasses of $P/exp$) exhibit some interesting internal structure. Three main results should be recalled from \cite{balcetal:88} (Chapter 5):

\begin{propos}\label{expspace}
There exist sets in \textit{EXPSPACE} not in $P/poly$. 
\end{propos}

Here \textit{EXPSPACE} utilises a Turing machine whose working tape, rather than being infinite, has length bounded by an exponential on the size of the input. There is no bound on time. This proposition above is proved by a non-trivial diagonalization of $P/poly$.

\begin{propos}\label{halting}
The {\em Halting Set} $H = \{0^n: \mbox{Turing machine with code $n$ halts on 0}\}$ is in $P/poly$.
\end{propos}

%A set is said to be {\em sparse} if its census is bounded by a polynomial. We also need to recall the concept of {\em tally set}: a set is said to be tally if it is a language over an alphabet of a single letter (we take this alphabet to be $\{0\}$). Tally sets are sparse (but not vice-versa). For each tally set $T$, $\chi_T$ (see Definition \ref{characteristic}) is defined relative to a single letter alphabet, e.g., $\Sigma = \{0\}$. The halting set $H$ above is tally.

%The most common characterization of $P/poly$ is given by the following statement, where by $P(S)$ we denote the class of sets decidable by deterministic Turing machines having access to the {\em oracle} set (or language) $S$:

%
%\begin{propos}\label{barsparsetally}
%$P/poly \ = \ \bigcup_{S \ sparse} \ P(S) \ = \ \bigcup_{S \ tally} \ P(S)$
%\end{propos}

%
%In \cite{balcetal:93} the following statement is offered as exercise to the reader (Chapter 5, Exercise 9). 

%
%\begin{propos}\label{tallyadvice}
%In polynomial time, tally oracle Turing machines and advice Turing machines are equivalent. \footnote{This Proposition can be generalized to arbitrary classes of time bounds.}
%\end{propos}

We will also need to treat prefix non-uniform complexity classes. For these classes we may only use prefix functions, i.e., functions $f$ such that $f(n)$ is always a prefix of $f(n+1)$. The idea behind prefix non-uniform complexity classes is that the advice given for inputs of size $n$ may also be used to decide smaller inputs.

\begin{defin}\label{prefixclass}
Let $\mathcal{B}$ be a class of sets and $\mathcal{F}$ a class of advice functions. The prefix advice class $\mathcal{B/F*}$ is the class of sets $A$ for which some $B \in \mathcal{B}$ and some prefix function $f \in \mathcal{F}$ are such that, for every length $n$ and input $w$, with $|w| \le n$, $w\in A$ if, and only if, $\langle w, f(n) \rangle \in B$. 
\end{defin}

Non-uniform classes are, indeed, relevant, by their impact in characterization theorems. For example, the class $P/poly$ is the class of sets decidable by the families of polynomial size circuits. They have been used in the many theories of analogue computation and hybrid systems (namely, see \cite{siegelmann:99,bournezcosnard:96}).

\subsection{Probabilistic classes}\label{sec:probclass}

For the probabilistic complexity classes it is a matter of some controversy whether Definition \ref{bar} is the appropriate definition of a non-uniform probabilistic class (eg., of $BPP/\mathcal{F}$). Notice that by demanding that there is a set $B \in BPP$ and a function $f \in \mathcal{F}$ (in this order!) such that $w \in A$ if, and only if, $\langle w, f(|w|) \rangle \in B$, we are demanding a fixed probability $\frac{1}{2} + \varepsilon$, $0 < \varepsilon < \frac{1}{2}$  (fixed by the Turing machine chosen to witness that $B \in BPP$) {\em for any possible correct advice}, instead of the more intuitive idea that the error $\gamma = \frac{1}{2} - \varepsilon$ only has to be bounded after choosing the correct advice. This leads to the following definitions for the specific complexity class $BPP$ that we will be using throughout this paper:

\begin{defin}\label{barbarpoly}
$BPP//poly$ is the class of sets $A$ for which a probabilistic Turing machine $TM$ clocked in
polynomial time, a function $f \in poly$, and a constant $0 < \gamma< \frac{1}{2}$ exist such that $TM$ rejects $\langle w, f(|w|) \rangle$ with probability at most $\gamma$ if $w \in A$ and accepts $\langle w, f(|w|) \rangle$ with probability at most $\gamma$ if $w \notin A$.
\end{defin}

\begin{defin}\label{barbarlog}
$BPP//log*$ is the class of sets $A$ for which a probabilistic Turing machine $TM$ clocked in polynomial time, a prefix function $f \in log*$, and a constant $0 < \gamma< \frac{1}{2}$ exist such that, for every length $n$ and input $w$ with $|w| \leq n$, $TM$ rejects $\langle w, f(n) \rangle$ with probability at most $\gamma$ if $w \in A$ and accepts $\langle w, f(n) \rangle$ with probability at most $\gamma$ if $w \notin A$.
\end{defin}

It can be shown that $BPP//poly = BPP/poly$, but it is unknown whether $BPP//log* \subseteq BPP/log*$. After the work of \cite{balcher:98}, we can assume without loss of generality that, for $P/log*$ and $BPP//log*$, the length of any advice $f(n)$ is exactly $\lfloor a \log n + b \rfloor$, for some $a, b \in \mathbb{N}$ which depend on $f$. The proof also applies to BPP//log*.

\section{The {\sc collider} computes $P/log*$ and $BPP//log*$}\label{sec:collidercomplexity}

\subsection{Coding the advice function as the unknown mass value}\label{code}

The aim is to code advice functions as quantities. We take a binary coding $c(a)$ for words $a$ by first converting $a$ to a string of 0s and 1s using a binary code for every letter in its alphabet $\Sigma$ (e.g. Unicode or ASCII). 
For example, we might suppose that the binary form of $a$ was $00101$. To find $c(a)$, replace every $0$ in its binary form by $100$ and every $1$ by $010$. Thus our example becomes $100100010100010$. If $a$ is empty then so is $c(a)$.
Now we must decide on a coding for $\log\!*$ advice, where the reader is reminded that the advice
 $f(n)$ can be used to decide all words $w$ with $|w|\le n$. We can recursively define $\mu(f)$
 as the limit of $\mu(f)(n)$,
 using our coding $c$, starting with
 \[
 \mu(f)(0)\,=\,0\cdot c(f(0))
 \]
 and using the following cases:
 \[
  \mu(f)(n+1)\,=\,\left\{\begin{array}{ll}
 \mu(f)(n)\,c(s) & \mathrm{if}\ n+1\  \mathrm{not\ a\ power\ of\ 2\ and }\ f(n+1)=f(n) s \\
   \mu(f)(n)\,c(s)\,001 & \mathrm{if}\ n+1\  \mathrm{a\ power\ of\ 2\ and }\ f(n+1)=f(n) s
 \end{array}\right. 
\]
The net effect is to add a 001 at the end of the codes for $n=1,2,4,8\dots$. 

Now take $w$ with $2^{m-1} < |w|\le 2^{m}$. We read the binary expansion until we have 
counted a total of $m+1$ 001 triples. Now the extra 001s are deleted, and
we can reconstruct $f(2^m)$, which can be used as advice for $w$. But how many
 binary digits from $\mu(f)$ must we read in order to reconstruct $f(2^m)$? 
We start with $|c(f(n))|\le L\ \log_2(n)+K$ (for some constants
$K$ and $L$) from the definition of log length. We must read at most 
$L\, m+K+3\,m$ digits when we add in the separators, so the number of digits is logarithmic in 
$|w|$. 

Of course, not every element of $(0,1)$ occurs as a $\mu(f)$. The reader may note that no dyadic rational can occur, as they would have to end in $0^{\omega}$ or $1^{\omega}$, and the triples $000$ and $111$ do not occur in any $\mu(f)$.  We will need a stronger result about which numbers can't be of the form $\mu(f)$:

\begin{propos}\label{teo2}
Let $f: \mathbb{N} \to\ \Sigma^*$ be an advice function. Given a dyadic rational $k/2^{n} \in (0,1)$, for integer $k$, $|\mu(f)-k/2^{n}| > 1/2^{n+5}$, for all $f$.
\end{propos}

\begin{proof}
A reasonably simple consequence of the fact that $\mu(f)$ has a binary expansion consisting
only of triples of the form 001 or 100 or 010. 
\end{proof}

Of course we have, using this coding technique,

\begin{propos}\label{exponentialbounds}
Given that the unknown mass $\mu$ is of the form $\mu(f)$ for some advice function $f$, at any step of the {\em bisection method} the lower and the upper bounds on the time of a single experiment are exponential in the word length.
\end{propos}

Why not just use the class $log$? The collider machine experiment allows us to read only $log(n)$ bits of the unknown mass in polynomial time because of the exponential time delay in the protocol. Assuming that $f$ is an advice in $log$, to encode the advice information, for $n = 0$ to $n = p$, we need $f(0)$, ..., $f(p)$. A way to encode this information requires the concatenation of all these sequences that amounts to a number of bits of the order of $n \log(n)$, greater than logarithmic size. This problem is avoided with the prefix advice $log*$.

\subsection{The error-free {\sc collider} can decide $P/log*$ in polynomial time.}\label{Plog}

In this subsection we will show that error-free colliders with an exponential AD-protocol can decide all the sets in $P/log*$ in polynomial time. We leave the more complicated problem of determining an upper bound for the complexity of the sets that these machines can decide in polynomial time (see \cite{beggsetal:09c}).

Let $A$ be a set in $P/log*$, and, by definition, let $B \in P$, $f \in log*$ be such that \[ w \in A \iff \langle w, f(|w|) \rangle \in B. \] Now take an AD-collider with the unknown mass set to $\mu(f)$, as described in subsection \ref{code}, and exponential time protocol. For a word $w$, if we can determine the advice $f(|w|)$ in polynomial time in $|w|$, then the Turing machine can determine $\langle w, f(|w|) \rangle$ in polynomial time, so the decision problem $w \in A$ can be solved in polynomial time in $|w|$. In turn, to determine $f(|w|)$ it suffices to show that we can read the first $\mathit{log(n)}$ binary places of the unknown mass $\mu(f)$ in polynomial time in $n$.

\begin{propos}\label{teo3} 
An error-free analogue-digital collider with an exponential AD-protocol can determine the first $\mathit{log(n)}$ binary places of the unknown mass $\mu(f)$ in polynomial time in $n$, where $n$ is the size of the input.
\end{propos}

\begin{proof}
Use the {\em bisection method} with the exponential protocol, which requires $k' \times \mathit{log(n)}$ collisions for some fixed constant $k'$, with collision $i$ requiring time $O(2^{k''i}) \subseteq O(2^{k \times \mathit{log(n)}}) = O(n^k)$, with $k = k'' \times k'$. Adding all these times gives a total amount of time polynomial in $n$. 
\end{proof}

\begin{theorem}\label{teo4}
An error-free analogue-digital collider with an exponential AD-protocol can decide $P/log*$ in polynomial time.
\end{theorem}

\begin{proof}
From the discussion earlier in this subsection, and the Proposition \ref{teo3}. 
\end{proof}

\subsection{The error-prone {\sc collider} with arbitrary precision can decide $P/log*$ in polynomial time.}\label{errorPlog}

Now we come to the error-prone arbitrary precision case, which is solved in almost exactly the same way.  The work lies in choosing the errors so that the same process actually works, and for that we need Proposition \ref{teo2}.

\begin{propos}\label{teo5} 
An error-prone arbitrary precision analogue-digital collider with a strictly exponential AD-protocol can determine the first $log(n)$ binary places of the unknown mass $\mu(f)$ in polynomial time in $n$, where $n$ is the size of the input.
\end{propos}

\begin{proof}
Use the {\em bisection algorithm} again. At the $i$-th stage in the bisection process (involving dyadic rationals with denominators $2^i$), we set the error in the mass $k/2^i$ of the test particle to be $1/2^{i+6}$, i.e.\ that the mass of the test particle lies in the interval $[k/2^i - 1/2^{i+6},k/2^i + 1/2^{i+6}]$.
By Proposition \ref{teo2}, the unknown mass cannot be in the error interval about the given dyadic rational, and thus the result of the experiment is the same as though the mass of the test particle was the given dyadic rational $k/2^i$. Also by Proposition \ref{teo2} the distance between the test mass and $\mu$ is at least $1/2^{i+6}$, so the experimental time is bounded by an 
exponential in $i$. 
Thus the first $log(n)$ binary places of the unknown mass $\mu(f)$ can be read in polynomial time in $n$. 
\end{proof}

\begin{theorem}\label{teo6}
An error-prone arbitrary precision analogue-digital collider with an exponential AD-protocol can decide $P/log*$ in polynomial time.
\end{theorem}

\begin{proof}
From the discussion earlier in this subsection, and the Proposition \ref{teo5}. 
\end{proof}

As the reader can notice, from the point of view of complexity classes so far considered, there is no such difference between {\em infinite precision} and {\em unbounded precision}. Readers tend to think that infinite precision is needed to achieve the power of such non-uniform complexity classes. But as far as we analysed in \cite{beggsetal:08c} and in the current paper, no such difference exists.

\subsection{The error-prone {\sc collider} with fixed {\em a priori} precision can decide $BPP//log*$ in polynomial time.}\label{BPPlog}

The mass of the test particle is set at some dyadic rational $m$ up to some dyadic fixed precision
$\varepsilon$. We will show that such machines may, in polynomial time, make probabilistic guesses of up to a $\mathit{log}$-number of
digits of the value of the unknown mass. We will then conclude that these collider machines decide exactly $BPP//\log*$. It is important to note that these machines are allowed to keep running after an experiment times out --- the result `time out' is returned by the physical oracle.

\begin{propos}\label{bpplog*}
For any real value $\delta<\frac{1}{2}$ and prefix function $f \in \log*$, there is an error-prone analogue-digital collider machine with fixed precision which obtains $f(n)$ in polynomial time with a probability of error of at most $\delta$.
\end{propos}

\begin{proof}
We encode the prefix function $f \in \log*$ as a real number $s\in (0,1)$  using the method of Section (\ref{code}). Our problem of how to obtain $f(n)$ in polynomial time becomes
how to read the first $ \lfloor a \log n + b \rfloor$ digits of $s$ in polynomial time
(for some constants $a,b$).

Suppose
that $\varepsilon$ is the a priori fixed error in measuring the mass of the proof
particle. We then set the unknown mass of our analogue-digital collider
machine at the value $\mu = \frac{1}{2} - \varepsilon/2 + s \varepsilon$, 
so that $\mu\in [\frac{1}{2} - \varepsilon/2,\frac{1}{2} + \varepsilon/2]$. 
Set a time limit $T$ on each experiment so that if $|m-\mu|\ge\varepsilon/4$
then we are guaranteed a result in time $T$. Then there is an $\eta\in(0,\varepsilon/4)$ 
so that the experiment runs out of time (exceeds the bound $T$) 
on the interval $(\mu-\eta,\mu+\eta)$ - the exact value of $\eta$ is not important, what is important is that the interval is symmetric about $\mu$. (In this proof we will ignore the end points of intervals, as they occur as random values of the test mass with probability zero.)

Our method for guessing digits of $s$ begins by commanding the collider to
shoot proof particles with mass $m=\frac{1}{2} \ (\pm \ \varepsilon)$ a number
$\zeta$ times. Then we get the following results with the stated probabilities:

1)\quad $m<\mu$, with probability $p=(\mu-\eta+\varepsilon-\frac12)/(2\,\varepsilon)$ -- this occurs a random number $\alpha$ times,

2)\quad $m>\mu$, with probability $q=(\frac12+\varepsilon-\mu-\eta)/(2\,\varepsilon)$ -- this occurs a random number $\beta$ times,

3)\quad out of time, with probability $r=2\,\eta/(2\,\varepsilon)$ -- this occurs a random number $\gamma$ times.

\noindent
By our assumption of independent uniform distribution for $m$ 
in the interval $[\frac{1}{2} - \varepsilon,\frac{1}{2} + \varepsilon]$, the resulting distribution is a trinomial with probability for $(\alpha,\beta,\gamma)$ being
\begin{eqnarray*}
\frac{p^\alpha\ q^\beta\ r^\gamma\ \zeta!}{\alpha!\ \beta!\ \gamma!}\quad,\quad \alpha+\beta+\gamma=\zeta\ .
\end{eqnarray*}
Now we consider the random variable $X=2\,\alpha+\gamma$, a combination chosen to cancel the number $\eta$ from the mean. This is fairly easily seen to have mean $\bar X=\zeta\,(\mu+\varepsilon-\frac12)/\varepsilon$, but its variance $\mathrm{Var}(X)$ is a little more difficult to find. 
\begin{eqnarray*}
\mathrm{Var}(X) 
&=& 4\,\mathrm{Var}(\alpha)+\mathrm{Var}(\gamma)+4\,\mathrm{Covar}(\alpha,\gamma) \cr
&=& 4\,\mathrm{Var}(\alpha)+\mathrm{Var}(\gamma)+4\,(\mathbb{E}(\alpha\,\gamma)-
\bar\alpha\,\bar\gamma) \ .
\end{eqnarray*}
A little calculation yields the expectation
\begin{eqnarray*}
\mathbb{E}(\alpha\,\gamma) = \zeta\,(\zeta-1)\,p\,r\ ,
\end{eqnarray*}
and substituting in $\bar\alpha=\zeta\,p$, $\bar\gamma=\zeta\,r$,
$\mathrm{Var}(\alpha)=\zeta\,p(1-p)$ and $\mathrm{Var}(\gamma)=\zeta\,r(1-r)$ gives
\begin{eqnarray*}
\mathrm{Var}(X)
&=& 4\,\zeta\,p(1-p)+\zeta\,r(1-r)-4\,\zeta\,p\,r \cr
&=& \zeta\  \frac{3\,\varepsilon+4\,s\,\varepsilon-4\,s^2\,\varepsilon-4\,\eta}{4\,\varepsilon} \cr
& \le & \zeta\  \frac{2\,\varepsilon+4\,s\,(1-s)\,\varepsilon}{4\,\varepsilon} \,\le\, \frac{3\,\zeta}4\ .
\end{eqnarray*}
Now use Chebyshev's inequality, which says
for every $\Delta>0$, 
\[\mathbb P ( |\bar X-X| > \Delta )  \le \frac{\mathrm{Var}(X)}{\Delta^2}.\] 
Putting $\Delta=x\,\zeta$ here gives
\begin{eqnarray*}
\mathbb P\Big( \left| \frac12+s-\frac{X}{\zeta}
\right| > x \Big) \le \frac{\mathrm{Var}(X)}{x^2\,\zeta^2}  \le \frac{3}{4\,x^2\,\zeta}\ .
\end{eqnarray*}
To read the $k$th binary place of $s$, according to the coding, it is sufficient to find $s$ accurate to
a value of $1/2^{k+5}$, and then the probability of error is
\begin{eqnarray*}
\mathbb P\Big( \left| \frac12+s-\frac{X}{\zeta}
\right| > 1/2^{k+5} \Big)  \le \frac{3\ 2^{2\,k+10}}{4\,\zeta}\ .
\end{eqnarray*}
To do this within probability of error $\delta$, we need a number of experiments
\begin{eqnarray*}
\zeta >  \frac{3\ 2^{2\,k+10}}{4\,\delta}\ .
\end{eqnarray*}
As $k$ is logarithmic in $n$, the result is polynomial time in $n$.
\end{proof}

The Proposition \ref{bpplog*} will guarantee us that for every fixed error $\varepsilon$ we can find an unknown mass that will allow for an CME to extract the desired information. It does not state that we can make use of any unknown mass independently of the fixed error $\varepsilon$. It can be shown, however, that if $\varepsilon$ is a dyadic rational, then we may guess $O(\log(n))$ digits of the unknown mass in polynomial time.

\begin{theorem}\label{thmbpplog*}
An  error-prone analogue-digital collider with an exponential AD-protocol can decide $BPP//\log*$ in polynomial time.
\end{theorem}

\begin{proof}
It is a consequence of the definition of the class $BPP//\log*$, taking in consideration the result of Proposition \ref{bpplog*}. See details in \cite{beggsetal:08c}. There is one minor modification to take into account the possibility of a `time out' result on an experiment. In the proof that we can (up to a given probability of failure) simulate a sequence of independent coin tosses, we need two mutually exclusive events whose union has probability one. To achieve this, we use the test particle crossing a given flag as one event, and the union of crossing the other flag and `time out' as the other event. That is, we
amalgamate two of the possibilities listed in the proof of Proposition \ref{bpplog*}.
\end{proof}

Measurements based on this more realistic assumption about precision lead to an apparent increase in power from $P/log*$ to $BPP//\log*$. This is because of the introduction of probabilities. (It is apparent because $P \subseteq BPP$ but we do not know if $P = BPP$.)  Since the class $BPP//\log*$ contains non-computable sets,  the machines can decide super-Turing languages.

\section{Measurability by the CME}\label{sec:measurability}
Fundamentally, our experiment CME tries to find an unknown mass $\mu$ by a sequence of comparisons with known masses $m$, which are rational multiples of some standard mass $M$. The aim is, for any given $n$, to find the first $n$ binary places of the number $\mu/M$. 
In the rest of the paper, we assume for convenience that
 $M = 1$. The time $T(|m|)$ taken for one run of the experiment using known mass $m$ is governed by an inequality $T(|m|) \ge K/|m-\mu|$. 

There are good and bad values of $\mu$ that place limitations on our experimental method. In the worst case $\mu$ has dyadic rational values. Here, any run of the experiment with $m=\mu$ will fail to give a result, and we will be left with $\mu$ being strictly inside a dyadic interval. But no matter how small the interval is, we will only be able to determine a fixed number of binary places. For example, if $\mu = 0\!\cdot\!001$ exactly, we will never be able to distinguish between $0\!\cdot\!001$ and $0\!\cdot\!00011\dots10$, where the 1s are repeated sufficiently many times before reaching a 0. 

Let us formulate just what we mean by a mass being measurable:

\begin{defin}\label{measurable}
A mass $\mu$ is said to be {\em measurable} by the CME if there exists a Turing machine $M$, equipped with a computable schedule $T$, such that it prints the first $n$ bits of $\mu$ on the output tape in less than $T(n)$ time steps without timing out in any query.
Similarly it is {\em feasibly measurable} if $T$ can be chosen to be time constructible (recall \cite{balcetal:88}).
\end{defin}

Note the importance of time constructibility here --- it allows the timing to be done by the Turing machine itself in real time, rather than employing another device as a timer. At another extreme, were we to allow $T$ to be non-computable, then any non-dyadic rational would be
measurable. 

What can be measured depends on the experimental procedure chosen to run the experiment. The following concept will prove useful:

\begin{defin}\label{univ55}
A Turing machine $M$ is said to be an {\em universal measuring procedure} for the CME if, for every measurable mass $\mu$, there exists a computable schedule $T$, such that $M$ equipped with $T$ measures $\mu$.
\end{defin}

\subsection{Most masses are measurable}
Here we examine some measurable masses, and show that, in the sense of measure theory, almost all masses are measurable.

\begin{lemma}\label{procedures}
Suppose that $\mu\in[0,1]$, and that the time taken to determine if
$|m-\mu|<\epsilon$ is $K/\epsilon$. We define a series of algorithms $E_r$,
for $r \in \mathbb{N}$, for the Turing machine
using the CME as oracle, as follows:

(a) For all $0 \leq p \leq 2^r$, fire a particle of mass $p/2^r$, using
waiting time $K\, 2^{2r+1}$; and so

(b) the experimental calls in $E_r$ take a total time $K\, 2^{3r+1}$.
\\
Then the measure of the set where algorithm $E_r$ fails to find $\mu$ to $r$
binary places is $\leq 2^{-r}$.
\end{lemma}

\begin{proof}
A set $F_r \subset [0,1]$, of length $\leq 2^{-r}$, is defined by \[ F_r
\,=\, \bigcup_{0 \leq p \leq 2^r} [\frac{p}{2^r} - \frac{1}{2^{2r+1}},
\frac{p}{2^r} + \frac{1}{2^{2r+1}}] \ \cap\ [0,1] \] If $x$ is not in the
set $F_r$, the algorithm $E_r$ will determine the first $r$ binary places of
$x$.
\end{proof}

\begin{cor}\label{nreschedulepp}
There are programs $P_k$ (for integer $k \geq 1$), with specified waiting
times, so that the following is true: There is a measure one set in $[0,1]$
so that, for all $\mu$ in this set, there is a $k$ so that, for all $n \geq
0$, $P_k$ will successfully determine the first $n$ binary digits of $\mu$.
\end{cor}

\begin{proof}
To find the first $n$ places of $\mu$, $P_k$ uses the algorithm $E_{n+k}$ as
described in
Proposition \ref{procedures}.
The set where $P_k$ may fail to find the first $n$ places is $F_{n+k}$, so
the set where  $P_k$ may fail for some $n$ is $\cup_{n \geq 1}\ F_{n+k}$,
and this set has measure $\le 2^{-k}$.
The set where, for all $k$, there is an $n$ so that $P_k$ may fail for that
$n$ is the measure zero set $\cap_{k\ge 1}\cup_{n\ge 1} F_{n+k}$.
\end{proof}

To emphasise the result of Corollary \ref{nreschedulepp}, if we choose $\mu\in[0,1]$
at random (with a uniform probability distribution), then $\mu$ will be measurable with probability one.

The following result is rather trivial, but it is important to point it out.

\begin{propos}\label{nreschedule}
There are programs $N_k$ (with integer $k \geq 1$), with specified waiting
times (say $T_k$), so that the following is true: For any non-dyadic $\mu
\in [0,1]$ and any $n \geq 0$, there is a $k$ so that program $N_k$ will
find the first $n$ binary places of $\mu$.
\end{propos}

\begin{proof}
The program $N_k$ runs the bisection procedure with a waiting time of $T_k$
for each experiment. As long as $\mu \in [0,1]$ is not dyadic, there is some
accuracy of measurement which will determine its first $n$ binary places.
\end{proof}

It is important to note that Propositions \ref{nreschedulepp} and \ref{nreschedule}
prove the existence of programs to find the first $n$ places; there is no
idea of being able find in advance which programs.  Note, too, that the
infinite sequence $T_1$, $T_2$, ..., $T_k$, ..., where $T_k$ is the time
needed to find experimentally the $k$-bit of $\mu$ is not recursively
enumerable. The Proposition \ref{nreschedule} can be read in the following
way: {\em if we know the unknown mass $\mu$ in advance --- as an oracle ---
then we can define a schedule of computation times for each bit of the
mass}.

Now we consider numbers which are easy to find using the bisection method.
Recall that a real number is algebraic of order $k$ if it is a root of a
polynomial of order $k$ with integer coefficients. The following result, due
to Liouville, is well known:

\begin{propos}\label{algebraic}
If $x$ is an algebraic number of order $k$ then, for all non-zero integers
$a$ and $b$ such that $x \neq  a/b$, there is a computable number $R(x) > 0$
so that
$$\Big| x-\frac{a}{b} \Big| \, \geq \, \frac{R(x)}{b^k} \ .$$
\end{propos}

\begin{propos}
If $\mu \in [0,1]$ is an algebraic number, and not a dyadic rational, then
there is a procedure to find $\mu$ so that the time schedule $T(n)$, to find
the $n$-th bit of $\mu$, is $\alpha\, n\, 2^{k n}$, for some computable constant
$\alpha$.
\end{propos}

\begin{proof}
Using the bisection procedure we need $O(n)$ experiments, each taking time
proportional to $2^{k n}$.
\end{proof}

\subsection{A characterisation of measurable masses}
We can characterise masses that can be measured quite precisely: Non-dyadic masses $\mu\in[0,1]$ can be written in the pattern form, where $u_k$ gives the number of digits in the $k$th group:
\begin{eqnarray}\label{pattern}
\mu = 0 \! \cdot \! \underbrace{1 \dots 1}_{u_1} \underbrace{0 \dots 0}_{u_2} \underbrace{1 \dots 1}_{u_3} \underbrace{0 \dots 0}_{u_4} \underbrace{1 \dots 1}_{u_5} \underbrace{0 \dots 0}_{u_6} \dots \quad
\mathrm{where}\ 
u_1\ge 0,\ u_i\ge 1\  (i\ge 2)\ .
\end{eqnarray}

\begin{propos}\label{nonmeasurability} For the CME with unknown mass $\mu$ (not a dyadic rational), written
according to the pattern (\ref{pattern}):

\noindent (1)\quad 
If $\mu$ is measurable by any program, then the sequence $u_k$ is bounded by a computable function.

\noindent (2)\quad 
If the sequence $u_k$ is bounded by a computable function, then $\mu$ is measurable by
the bisection method. 
\end{propos}

\begin{proof}
First note that the digit at the end of the block labelled by $u_k$ is in the $a_k$-th position, where $a_k=u_1+\dots+u_k$.
To make the method obvious we use an example with $u_1=3$, $u_2=2$, $u_3=4$, $u_4=3$, etc.
\begin{eqnarray*}
\mu=0 \! \cdot \! 11100111100011\dots
\end{eqnarray*}
To determine all digits up to the $a_3=9$th digit any program must have successfully run the experiment with test masses
in the intervals $[\mu^-,\mu)$ and $(\mu,\mu^+]$, where $\mu^\pm$ are the $a_3$ digit dyadic rationals,
differing only in the last position
\begin{eqnarray*}
\mu^-=0 \! \cdot \! 111001110\ ,\quad \mu^+=0 \! \cdot \! 111001111\ .
\end{eqnarray*}
Then we have the inequalities
\begin{eqnarray*}
2^{-a_3} \le |\mu-\mu^-| \le 2^{1-a_3}\ ,\quad 2^{-a_4-1} \le |\mu-\mu^+| \le 2^{-a_4}\ .
\end{eqnarray*}
(1)\quad For the first $a_k$ digits of $\mu$ to be determined, we must perform at least one experiment
of duration at least $2^{a_{k+1}}\, K$, where $K$ is a constant. If $\mu$ is measurable, there must
be a computable function $T$ so that $T(a_k)\ge 2^{a_{k+1}}\,K$, and from this we determine
the following formula, from which a computable bound for the sequence $u_k$ can be derived:
\begin{eqnarray*}
2^{u_{k+1}}\ \le \ 2^{-a_k}\, T(a_k)/K\ .
\end{eqnarray*}
(2)\quad Note that in the previous discussion, the numbers $\mu^\pm$ are the last two numbers queried in the binomial method for finding the first $a_3$ digits of $\mu$. In general to find the first $a_k$ digits of $\mu$ by the binomial method, we need $a_k$ experiments, each of duration at most $2^{a_{k+1}+1}\, K$. If the sequence $u_n$ is computable, this gives a computable schedule for finding $\mu$. 
\end{proof}

\begin{cor}\label{bisectionuniversal}
The Turing machine equipped with the bisection algorithm is a universal measuring procedure
(see Definition \ref{univ55}). 
\end{cor}

\subsection{Results on non-measurable masses}
Now let us consider further numbers which are difficult to find using our
method.
For convenience, we recall the specification of a non-computable function 
$\mathrm{beaver}: \mathbb{N} \to \mathbb{N}$ called the \textit{busy beaver function} (cf. \cite{cooper:04}). 

\begin{defin}\label{dominates}
Let $f, \ g: \mathbb{N} \to \mathbb{N}$ be total functions. We say that $g$ {\em dominates} $f$ if there exists a natural number $p$, called an order, such that, for all $n$ such that $n > p$, we have $g(n) > f(n)$. If $\mathcal{F}$ is a set of such functions, we say that $g$ {\em dominates} $\mathcal{F}$ if $g$ dominates $f$, for all $f \in \mathcal{F}$.
\end{defin}

\begin{defin}\label{beaver}
Let $beaver: \mathbb{N} \to \mathbb{N}$ be the total function defined by: $beaver(0) = 0$ by convention; $beaver(n)$ is the maximum output for input $0$ among all Turing machines with $n$ states that halt on input $0$.
\end{defin}

The $beaver$ is a totally defined function because for all $n$ there exists a Turing machine with $n$ states that halts and, consequently, can produce a string of $1$s on the output tape that can be interpreted in unary or  binary, according to convention. 

This function is due to Tibor Rad\'o \cite{rado:62} and was one of the first well-defined non-computable total functions. Unsurprisingly, the function is complicated and the growth of the function is considered an open problem \cite{brady:94}. 

\begin{theorem}\label{beaverdominates}
The function $beaver$ dominates all total computable functions.
\end{theorem}

Notice that for every total computable function $f$ there in an order $p_{f}$, depending upon the function $f$, such that from $p_{f}+1$ onward the busy beaver grows faster than $f$.

\begin{theorem}\label{nonmeasurable}
There are uncountably many values $\mu\in[0,1]$ which are not measurable by the CME.
\end{theorem}
\begin{proof}
We take all $\mu$, defined by Pattern \ref{pattern}, made from all
possible choices of the following values of each $k$:
\begin{eqnarray*}
u_k \, = \, \left\{ \begin{array}{c} \mathrm{beaver}(k) \\
\mathrm{beaver}(k)+1\end{array} \right.
\end{eqnarray*}
Because of the choice, there are uncountably many such $\mu$. By Proposition \ref{nonmeasurability},
all of them are non-measurable. 
\end{proof}

Can we decide if a mass is measurable? To be more precise, we could imagine using a program on a Turing machine using the CME as oracle to decide this. However, we have the following negative result:

\begin{propos}\label{undecmeasurable}
There is no program running on a Turing machine using the CME as oracle which can decide in finite time if a mass is measurable.
\end{propos}

\begin{proof}
In the given amount of time the program runs for, the CME can only find the mass $\mu$ within a given open interval. 
Any mass within this open interval would produce the same result for the program. However this interval contains (after some initial segment of $\mu$) both masses with infinite endings 
of their decimal expansions of the form coded earlier (and therefore effectively measurable) and those with busy beaver endings (and therefore not effectively measurable).
\end{proof}

\section{An uncertainty principle in classical mechanics}\label{sec:uncertainty}
Using an algorithmic theory of measurement, we have shown that for an experiment in classical dynamics, what is measurable depends on:

\noindent 1.  \textit{Time}: To buy accuracy you have to pay with time, and the budget for time is controlled by schedules.

\noindent 2. \textit{Equipment}: The apparatus for the CME, and the physical theory that governs it, imposes its own limits on how much time is required for a given experiment. 

\noindent 3. \textit{Procedure}: For the CME there is a universal experimental procedure (the bisection method), which can obtain all the results that can be measured by all experimental procedures. 

The limitation on the apparatus in (2) can be phrased as a sort of uncertainty principle, using equation (\ref{time}) to make an inequality with $\Delta\mu=|m-\mu|$ being the uncertainty in mass, and
$\Delta t$ being the time necessary to perform an experiment, where $u$ is the input velocity and $r$ is the distance from the unknown mass to the flags:
\begin{eqnarray*}
\frac{\mathrm{lower\ bound\ on\ masses}\times r}{u} \, \le\, 
\Delta\mu \times \Delta t\, \le\, \frac{\mathrm{2\ upper\ bound\ on\ masses\times r}}{u}\ .
\end{eqnarray*}
$\Delta \mu \times \Delta t$ is a product of the type $\Delta E \times \Delta t$, where $E$ denotes {\em energy}, which is quite well known both in classical and 
 quantum physics (see \cite{hamming:89}), and it is not considered a purely quantum relation.

What happens when we combine this trade-off with the computable schedules of the experimental procedure, which bound the time? 
Do we find notions of the limits of measurement or quanta? Yes, in a way. Imagine a resource allocation problem that many readers will be familiar with, that of a research council buying time from a research group. With no schedule, the council keeps on paying for a result that will be delivered `eventually'. With a schedule, there is a table of what results will be delivered by various deadlines, and failure to meet a deadline means that the grant is terminated. 

As far as our CME measuring mass is concerned, we may as well assume for simplicity that we are running the bisection method, since we have shown that it is universal. 
We have shown that there are physical masses $\mu$ for which every possible computable schedule will eventually fail. 
A particular computable schedule will fail at these masses, and more. That failure will consist of not being able to determine the $n$th digit in time $T(n)$ (after determining the previous $n-1$ digits). But then, using that schedule, we would not be able to distinguish any masses between the mass given by the determined $n-1$ digits followed by 0, and the mass given by the $n-1$ digits followed by 1. Thus, for that mass, and for the amount of effort that society (in the form of the research council) is prepared to devote to finding it, there is an effective `quantum' or limit of observability. That limit can either be expressed as a limit on the number of places, or in terms of an interval. Of course, everything, in particular a notion of quantum, depends on  $\mu$ and $T$. It is quite possible that the experimental procedure will continue indefinitely, as the CME 
for that mass and schedule keeps meeting every deadline.

\section{Conclusions}\label{conclusions}

We are developing a methodology and mathematical theory to examine how data is represented and computed using physical systems. Our primary tasks are to study computation by (i) physical systems in isolation and (ii) physical systems combined with algorithms. Our objectives are foundational rather than technological. The main ideas and results of this paper play an influential role in our research programme. The methodological principles of Section \ref{Computation_and_physical_systems}, especially Principles 5 and 6, allow us to pursue questions of interest in computation, physics and philosophy. We believe the CME exemplifies technical ideas and properties that have very wide application.

In summary, the CME focusses attention on the idea that the communication between the Turing machine and an external physical device is complicated and that the concept of a \textit{protocol} is important and essential. In particular, physical oracles \textit{must} take time to consult that depend exponentially upon accuracy.  Applying Principle 5, we showed that Turing machines and protocols boost computational power beyond the Turing barrier (Section \ref{sec:collidercomplexity}). Applying the new Principle 6, we showed that Turing machines and protocols reduce the power of experiments to measure the classical continuum (Sections \ref{sec:measurability} and \ref{sec:uncertainty}). We will reflect on time, our earlier experiment SME, and uncertainty.

\subsection{Time and a conjecture}\label{conjec}

In his essays \cite{bachelard:40,bachelard:34}, Bachelard stresses the fact that accuracy in a measurement in Physics is related to time: more accurate measurements consume more time. In our setting, the protocol must cope with the time (a) to settle the parameters of the experimental equipment and (b) to accomplish the experiment to the desired accuracy. The Turing machine's query denotes the actual values of the parameters and the desired accuracy is given by the size of the query. Now, while (a) may depend upon the experimenter, (b) depends only on the physical theory $T$ specifying the experiment: the protocol's schedule is a measure of the time needed, or allowed, to perform the experiment and retrieve the result. In the collider experiment, the time needed for a collision experiment is exponential in the size of the Turing machine query. This is a new and remarkable fact about the protocol because it seems to be common; indeed, exponential time seems to be the norm.

The new features of the CME, which contrast with our earlier work on SME, led us to consider several standard experiments in Physics, measuring distance, inertial mass, resistance, temperature, the ratio $e/m$ of an elementary particle in the Coulombian field applying classic or quantum methods, and the Brewster angle in optics. In all these experiments we found that the {\em the time needed for an experiment is exponential in the size of the Turing machine query}. 
 
In Inequality \ref{ctimesim}, the lower bound for the experimental time of the CME can be considered as less controversial than the upper bound. In our fragment of physical theory, we have a somewhat idealised world and it is likely that trying to make the theory more ``realistic'' would increase the experimental time. In other words, the lower bound of Inequality \ref{ctimesim} is likely to remain, though adding ``realism'' might cast doubt on the upper bound. 

We shall make a conjecture about the behaviour of Turing machines using physical oracles, for which we plan to publish more evidence in due course. It is based on the idea that the lower bound to experimental time similar to that in Inequality \ref{ctimesim} is a universal feature of ``realistic'' theories. 

\begin{conjec}\label{bctconjecture} 
For ``realistic'' physical theories, using an experiment as a physical oracle requires an exponential time protocol.
\end{conjec}

We have begun the refinement and formalisation of this conjecture and the exploration of its consequences as a general mathematical property. If the conjecture is true generally then the project of finding physical systems which allow us to measure some physical quantity ever more accurately \textit{and} ever more efficiently --- e.g., allowing us to halve the error without doubling the time taken --- is condemned to failure. 

This exponential lower bound limits the rate that the Turing machine can ``extract'' information from the physical oracle and so limits any computational power which may be added to the Turing machine as a result of being connected to the physical oracle.

\subsection{Comparing the CME and SME}\label{cme&sme}

The scatter machine experiment SME was introduced in \cite{beggstucker:07b} and studied as an example of experimental computation using the principles in Section \ref{Computation_and_physical_systems}. We showed that the experiment could measure or compute non-computable numbers in $[0, 1]$. Among its principal features are the bisection method and non-deterministic discontinuous physical behaviour. The theory of oracles began in \cite{beggsetal:08c} where we introduced the ideas of (a) infinite, unlimited and finite precision and (b) protocols. Again, the theme was computational power: we introduced Principle 5 and characterised the power using non-uniform complexity classes. However, there are some important differences in our results about the SME and CME.

In the scatter machine, the time needed to perform a single experiment is constant. The protocol for the scatter machine we used was polynomial but this is quite an arbitrary assumption to cover the time taken to set up the cannon to the desired accuracy, fire it, and observe the result: it assumes that setting up an experiment to a higher accuracy takes more time. However, once the cannon position is set, the experiment is concluded in a few seconds.

A second difference is in the computational power the two experiments as oracles. The results in \cite{beggsetal:08c,beggsetal:08f} and here are compared in the following table.

\begin{center}

\begin{tabular}{||c|c|c||} \hline
\multicolumn{3}{|c|}{{\sc Computational Power}} \\ \hline\hline
{\em Experimental oracle} & {\em precision} & {\em complexity class lower bound} \\ \hline
& infinite & $P/poly$ \\ \cline{2-3}
{\bf scatter machine experiment} & unbounded & $BPP//poly = P/poly$ \\ \cline{2-3}
& fixed & $BPP//log*$ \\ \hline
& infinite & $P/log*$ \\ \cline{2-3}
{\bf collider machine experiment} & unbounded & $P/log*$ \\ \cline{2-3}
& fixed & $BPP//log*$ \\ \hline\hline
\end{tabular}

\end{center}

The transition from the experiment SME to the experiment CME implies a loss of computational power. In the infinite precision and unbounded precision cases, the computational power, in polynomial time, falls from $P/poly$ to $P/log*$. However, in the finite precision case, where the protocols have no influence, the computational power, in polynomial time, is the same $BPP//log*$. More subtly, in the CME the computational power rises from the infinite and unbounded precision $P/log*$ to finite precision $BPP//log*$ because the protocol is exponential and this does not effect the stochastic nature of the finite precision calculations. In the SME the computational power falls from the infinite and unbounded precision $P/poly$ to finite precision $BPP//log*$ because the protocol is polynomial and this is a stronger condition.

Now the SME involves discontinuous behaviour in the speed of scattered particles as a function of the unknown position, which leads to non-determinism at the vertex of the wedge (at the vertex, the speed jumps from a finite non-zero value to the negative of that value) while the CME is continuous (the speed of a test particle falls to zero as its mass approaches the unknown mass from below and then increases again, as the mass deviates from the unknown mass from above, but with opposite sign).

Does the SME falsify the Conjecture \ref{bctconjecture}? The scatter machine has an upper bound which is not of the form in Inequality \ref{ctimesim} --- in fact we could take a constant upper bound. Recently, we have  revisited the SME and found that introducing more ``realistic'' assumptions into the theory used --- rounding the vertex, for example --- not only removes this constant upper bound, but institutes a lower bound of the form of Inequality \ref{ctimesim}. This extended SME has an exponential protocol and the same computational power as CME. See \cite{beggsetal:09g} for a full account including an extended comparison table. This leads us to conjecture:

\begin{conjec}\label{physicalcurchturingthesis}
For ``realistic'' physical theories, using an experiment as a physical oracle with an exponential time protocol boosts the power of Turing machines to $P/log*$ for infinite and unbounded precision and to $BPP//log*$ for fixed precision.
\end{conjec}

Thus, our experiences with a portfolio of physical oracles suggests that the CME introduces limiting results of wide relevance to the physical sciences. 

\subsection{On measurement}\label{uncertainty}
In \textit{The Science of Mechanics}, Ernst Mach observes: ``The laws of impact were the occasion of the enunciation of the most important principles of mechanics, and furnished also the first examples of the application of such principles.''  It seems that the same laws, by governing the CME, introduce some new properties of the concept of measurement in mechanics. This type of experiment to measure mass is at the heart of mechanics --- a generalization of the collider experiment can be used to measure the mass of a star or of a planet, measures that cannot be done with the balance scale. 

Our measurement of {\em inertial mass} is {\em fundamental}, not {\em derived}:  according to Hempel \cite{hempel:52}): 

``By derived measurement we understand the determination of a metric scale by means of criteria which presuppose at least one previous scale measurement. It will prove helpful to distinguish between derived measurement by stipulation and derived measurement by law. The former consists in defining a ``new'' quantity by means of others, which are already available; it is illustrated by the definition of the average speed of a point during a certain period of time as the quotient of the distance covered and the length of the period of time. Derived measurement by law, on the other hand, does not introduce a ``new'' quantity but rather an alternative way of measuring one that has been previously introduced.''

We do not use a scale of distance, neither a scale of time. The algorithm only makes a fundamental direct measurement of mass (to see a deep discussion into this subject, see \cite{beggsetal:09f}).

Relevant to our context of derived measurements, Eddington writes on the fine-structure constant in \cite{eddington:33}:

``There has been much discussion whether the true value is 137.0 or 137.3; both values claim to be derived from observation. The latter, called the ``spectroscopic value'', is preferred by many physicists. It is, however, misleading to call these determinations \textit{observational values}, for the observations are only a substratum; the spectroscopic value in particular is based on a rather complex theory and is certainly not to be treated as a ``hard fact'' of observation.''

Actually, in most situations measurement is made by comparisons between observables, so what we describe applies not only to the measurement of {\em mass}, but widely in the physical sciences.

The idea of uncertainty is associated with \textit{Heisenberg's indeterminacy principle} and is the subject of an enormous philosophical discussion about quantum mechanics. Unusually, in \cite{popper:50a,popper:50b}, Popper struggles with the notion of indeterminacy in classical and quantum mechanics. The nature of our uncertainty is different, but its philosophical implications are similar.

Principle 6 is intriguing and may prove influential. Turing's analysis of people representing information symbolically and following a fixed procedure is an example of an anthropomorphic principle underpinning a scientific theory: models of computability rooted in human action are relevant to computing technologies of the past, present or future.  It seems to us to be a beautiful idea to model the experimenter following an experimental procedure as a new form of program based on a physical theory $T$ that can be coded as a Turing machine; and this idea resonates with the prominent and essential role of computers in performing experiments. It leads us believe that measurability in Physics is subject to \textit{laws} which are the effects of the limits of computability and computational complexity.  Our algorithmic model of experiments imposes limitations on the physics we used to describe it. Not all masses can be known, not because of the limitations in measurements due to experimental errors, but because of essentially internal logical limitations of the theory. The mathematics of computation theory does not allow the reading of bits of physical quantities beyond a certain limit. Quantities cannot be measured with infinite precision, not because of the limitations of the physical apparatus but, more deeply, because of computational reasons.  These unmeasurabilities allow for the definition of quanta of energy in the classical physical world. 

We believe that the computational model of experimental measurement, here represented by the collider, demonstrates the existence of new and fundamental epistemic constraints in physics.

Edwin Beggs, Jos\'e F\'elix Costa and John Tucker would like to thank EPSRC for their support under grant EP/C525361/1. The research of Jos\'e F\'elix Costa is also supported by FEDER and FCT Plurianual 2007 .

\bibliographystyle{plain}
\bibliography{references}

\newpage

\end{document}